\documentclass[10pt]{article}
\usepackage{graphicx,float,latexsym,amsmath,amsfonts,amssymb,subeqnarray,colordvi,epstopdf}
\usepackage[latin1]{inputenc}
\usepackage{a4wide}

\def\R{\mathbb{R}}

\newtheorem{theorem}{Theorem}

\newtheorem{lemma}[theorem]{Lemma}

\newenvironment{proof}{\noindent {\it Proof}~}{}

\title{Convergence of the Modified Craig--Sneyd scheme for two-dimensional 
convection-diffusion equations\\ with mixed derivative term}
\author{Karel~J.~in 't Hout\footnote{Department of Mathematics and Computer Science,
University of Antwerp, Middelheimlaan 1, B-2020 Antwerp, Belgium.
\mbox{Email}: \texttt{\{karel.inthout,maarten.wyns\}@uantwerpen.be}.}
~and Maarten Wyns\footnotemark[\value{footnote}]
}
\date{\today}

\begin{document}

\maketitle

\begin{abstract}
\noindent
We consider the Modified Craig--Sneyd (MCS) scheme which forms a prominent time stepping method of the Alternating Direction Implicit type for multidimensional time-dependent convection-diffusion equations with mixed spatial derivative terms.
Such equations arise often, notably, in the field of financial mathematics.
In this paper a first convergence theorem for the MCS scheme is proved where the obtained bound on the global temporal discretization errors has the essential property that it is independent of the (arbitrarily small) spatial mesh width from the semidiscretization.
The obtained theorem is directly pertinent to two-dimensional convection-diffusion equations with mixed derivative term.
Numerical experiments are provided that illustrate our result.
\end{abstract}
\vspace{0.2cm}\noindent
\vspace{3mm}
\normalsize

\setcounter{equation}{0}
\section{Introduction}\label{intro}

Semidiscretization by finite difference or finite volume methods of initial-boundary value problems for multidimensional time-dependent convection-diffusion equations leads to large systems of stiff ordinary differential equations (ODEs),
\begin{equation}
U'(t)=F(t,U(t)) \quad (0 \leq t \leq T), \quad U(0)=U_{0},
\label{eq:ODE}
\end{equation}
with given vector-valued function $F: [0,T]\times \R^m \rightarrow \R^m$ and given initial vector $U_{0}\in \R^m$, where integer $m\ge 1$ is the number of spatial grid points. 
For the effective time discretization of such semidiscrete systems, operator splitting schemes of the Alternating Direction Implicit (ADI) type are widely considered. 
Recently, ADI schemes have been studied for the situation where mixed spatial derivative terms are present in the convection-diffusion equation.
Mixed derivative terms are ubiquitous, notably, in the field of financial mathematics.
There they arise due to correlations between the underlying stochastic processes.
In the past years a variety of positive results has been derived in the literature on the {\it stability}\, of ADI schemes in the presence of mixed derivative terms, see e.g.~\cite{CS88,IHM11,IHM13,IHW07,IHW09}.
However, a useful {\it convergence} analysis relevant to this important situation is still in its infancy.
Our aim in the present paper is to contribute to this analysis.

Let the semidiscrete function $F$ be decomposed as
\begin{equation}
\label{eq:Decomposition}
F(t,v) = F_{0}(t,v) + F_{1}(t,v) + \cdots + F_{k}(t,v) \quad (0 \leq t \leq T,~v \in \R^m),
\end{equation}
where $F_{0}$ represents all mixed spatial derivative terms and $F_{j}$, for $1 \leq j \leq k$, represents all spatial derivative terms in the $j$-th direction.
For the time discretization of \eqref{eq:ODE} we consider a prominent scheme of the ADI type.
Let $\theta > 0$ be a given parameter, let $\Delta t>0$ be a given time step size and set $t_{n}=n \Delta t$ for integers $n\ge 0$.
Then the \textit{Modified Craig--Sneyd (MCS) scheme} defines approximations $U_{n}$ to $U(t_{n})$ successively for $n=1,2,3,\ldots$ with $n \Delta t \le T$ through
\begin{equation}
\label{eq:MCS}
\left\{\begin{array}{l}
Y_0 = U_{n-1}+\Delta t\, F(t_{n-1},U_{n-1}), \\\\
Y_j = Y_{j-1}+\theta\Delta t \left(F_j(t_n,Y_j)-F_j(t_{n-1},U_{n-1})\right),
\quad j=1,2,\ldots,k, \\\\
\widehat{Y}_0 = Y_0+ \theta \Delta t \left(F_0(t_n,Y_k)-F_0(t_{n-1},U_{n-1})\right),\\\\
\widetilde{Y}_0 = \widehat{Y}_0+ (\tfrac{1}{2}-\theta )\Delta t \left(F(t_n,Y_k)-F(t_{n-1},U_{n-1})\right),\\\\
\widetilde{Y}_j = \widetilde{Y}_{j-1}+\theta\Delta t \,(F_j(t_n,\widetilde{Y}_j)-F_j(t_{n-1},U_{n-1})),
\quad j=1,2,\ldots,k, \\\\
U_n = \widetilde{Y}_k.
\end{array}\right.
\end{equation}
The MCS scheme \eqref{eq:MCS} was introduced by in 't Hout \& Welfert \cite{IHW09}.
It starts with an explicit Euler predictor stage, which is followed by $k$ implicit but unidirectional corrector stages.
Then a second explicit stage is performed, followed again by $k$ implicit unidirectional corrector stages.
In the special case where $\theta = \frac{1}{2}$ one obtains the Craig--Sneyd (CS) scheme, proposed in \cite{CS88}.
Note that the $F_0$ part, representing all mixed derivative terms, is always treated in an explicit manner by \eqref{eq:MCS}.

In \cite{IHM11,IHM13,IHW09,M14} positive stability results were proved for the MCS scheme under convenient lower bounds on $\theta$, guaranteeing unconditional stability in the von Neumann sense pertinent to various classes of multidimensional convection-diffusion equations with mixed derivative terms.
A relevant convergence analysis is still open in the literature.
It can be verified by standard arguments that if natural stability and smoothness assumptions hold, then the MCS scheme is convergent of order two for fixed, nonstiff ODE systems.
It is well-known in the literature, however, that this standard convergence analysis of time stepping schemes has limited relevance for the application to semidiscrete systems \eqref{eq:ODE}.
In this analysis, the size of the error constant in the obtained bound for the global temporal discretization errors may become arbitrarily large as the spatial mesh width from the semidiscretization tends to zero ($m\rightarrow \infty$), which renders this bound impractical.

In the present paper we shall derive a first useful convergence bound for the MCS scheme that is directly relevant to two-dimensional convection-diffusion equations with mixed derivative term.
Our analysis is inspired by that of Hundsdorfer \cite{Hun92,Hun02}, cf.~also Hundsdorfer~\&~Verwer \cite{HV03}, for operator splitting schemes applied to multidimensional convection-diffusion-reaction problems without mixed derivative terms and it extends our recent work~\cite{IHW14} for the Hundsdorfer--Verwer (HV) scheme.

\setcounter{equation}{0}
\section{Convergence analysis}\label{sec:results}

\subsection{Preliminaries}
Assume that
\[
F(t,v) = Av + g(t)\quad {\rm and} \quad F_j(t,v )=A_jv +g_j(t) \quad (0 \leq t \leq T,~ v \in \mathbb{R}^{m},~ 0\le j\le k),
\]
where $A$, $A_{j}$ ($0\leq j \leq k$) are given real $m \times m$-matrices and $g$, $g_{j}$ ($0 \leq j \leq k$) are given real $m$-vector valued functions.
Let $I$ denote the $m \times m$ identity matrix.
For convenience, define the matrices
\begin{equation*}
Z=\Delta t A,\quad Z_{j} = \Delta t A_{j}, \quad Q_{j} = I - \theta Z_{j} \quad (0 \leq j \leq k),\quad P=Q_{1}Q_{2}\cdots Q_{k}.
\end{equation*}
Consider the naturally scaled inner product $(v,w)=\frac{1}{m} v^{\rm T} w$ for $v, w \in \R^{m}$ with induced vector and matrix norms $\Vert \cdot \Vert_{2}$.
We shall assume that
\begin{equation*}
\label{eq:AssumptionA}
(A_{j}v, v) \leq 0 \quad {\rm whenever}~ v \in \R^{m},~ 1\leq j \leq k.
\end{equation*} 
This assumption is often fulfilled when dealing with semidiscrete systems stemming from time-dependent convection-diffusion equations, cf.~e.g.~\cite{Hun92,Hun02,HV03}. 
It implies that the $Q_{j}$ and $P$ are invertible and 
\begin{equation}
\label{eq:PQbounded}
\Vert Q_{j}^{-1} \Vert_{2} \leq 1\quad (1 \leq j \leq k),\quad \Vert P^{-1} \Vert_{2} \leq 1.
\end{equation}

\subsection{Error recursion}
For the convergence analysis, we consider along with \eqref{eq:MCS} the perturbed scheme
\begin{equation}
\label{eq:MCSPerturbed}
\left\{ \begin{array}{lr}
     	Y_{0}^{*} = U_{n-1}^{*} + \Delta t\, F(t_{n-1},U_{n-1}^{*}) + \rho_{0},& \\\\
 	 	Y_{j}^{*} = Y_{j-1}^{*} + \theta\Delta t \left( F_{j}(t_{n},Y_{j}^{*}) - F_{j}(t_{n-1},U_{n-1}^{*}) \right) + \rho_{j}, & j = 1,2,\ldots,k, \\\\
 	 	\widehat{Y}_{0}^{*} = Y_{0}^{*} + \theta \Delta t \left( F_{0}(t_{n},Y_{k}^{*}) - F_{0}(t_{n-1},U_{n-1}^{*}) \right) + \widehat{\rho}_{0}, & \\\\
 	 	\widetilde{Y}_{0}^{*} = \widehat{Y}_{0}^{*} + ( \tfrac{1}{2} - \theta ) \Delta t \left(F(t_{n},Y_{k}^{*}) - F(t_{n-1},U_{n-1}^{*}) \right) + \widetilde{\rho}_{0}, & \\\\
 	 	\widetilde{Y}_{j}^{*} = \widetilde{Y}_{j-1}^{*} + \theta\Delta t\, ( F_{j}(t_{n},\widetilde{Y}_{j}^{*}) - F_{j}(t_{n-1},U_{n-1}^{*})) + \widetilde{\rho}_{j}, & j = 1,2,\ldots,k, \\\\
 	 	U_{n}^{*} = \widetilde{Y}_{k}^{*}.
     	\end{array} \right.
\end{equation} 
Here $\rho_{j}, \widetilde{\rho}_{j} \in \R^m$ ($0\le j \le k$) and $\widehat{\rho}_{0}\in \R^m$ denote arbitrary given perturbation vectors.
These perturbations may depend on the step number $n$.
For ease of presentation, this is omitted in the notation.
In the following we derive a useful formula for the error
\[
e_{n}=U_{n}^{*} - U_{n}.
\]
Define the auxiliary variables
$$
\varepsilon_{j} = Y_{j}^{*} - Y_{j}, \quad \widetilde{\varepsilon}_{j} = \widetilde{Y}_{j}^{*} - \widetilde{Y}_{j}\quad (0 \leq j \leq k)\quad {\rm and}\quad
\widehat{\varepsilon}_{0} = \widehat{Y}_{0}^{*} - \widehat{Y}_{0}.
$$ 
From (\ref{eq:MCS}), (\ref{eq:MCSPerturbed}) one directly obtains
\begin{eqnarray*}
\varepsilon_{0} &=& e_{n-1} + \Delta t A e_{n-1} + \rho_{0} \\
			&=& (I + Z) e_{n-1} + \rho_{0}, \\
\varepsilon_{j} &=& \varepsilon_{j-1} + \theta \Delta t (A_{j} \varepsilon_{j} - A_{j} e_{n-1} ) + \rho_{j}, \qquad 1 \leq j \leq k.
\end{eqnarray*}
The latter equation can readily be rewritten as
\begin{equation}\label{eq:vareps}
\varepsilon_{j} = e_{n-1} + Q_{j}^{-1} ( \varepsilon_{j-1} - e_{n-1} + \rho_{j} ), \qquad 1 \leq j \leq k.
\end{equation}
Next, 
\begin{eqnarray*}
\widehat{\varepsilon}_{0} &=& \varepsilon_{0} + \theta Z_{0} ( \varepsilon_{k} - e_{n-1} )  + \widehat{\rho}_{0} \\
    &=& (I + Z - \theta Z_{0} ) e_{n-1} + \theta Z_{0} \varepsilon_{k} + \rho_{0} + \widehat{\rho}_{0}, \\
\widetilde{\varepsilon}_{0} &=& \widehat{\varepsilon}_{0} + (\tfrac{1}{2} - \theta) Z ( \varepsilon_{k} - e_{n-1} ) + \widetilde{\rho}_{0} \\
	&=& (I+ (\tfrac{1}{2}+\theta ) Z - \theta Z_{0} ) e_{n-1} + (\theta Z_{0} + (\tfrac{1}{2}-\theta) Z ) \varepsilon_{k} + \rho_{0} + \widehat{\rho}_{0} + \widetilde{\rho}_{0}
\end{eqnarray*}
and analogously to \eqref{eq:vareps} there holds
\begin{equation}\label{eq:varepstilde}
\widetilde{\varepsilon}_{j} = e_{n-1} + Q_{j}^{-1} ( \widetilde{\varepsilon}_{j-1} - e_{n-1} + \widetilde{\rho}_{j} ), \qquad 1 \leq j \leq k.
\end{equation}
Using \eqref{eq:varepstilde} together with the obtained expression for $\widetilde{\varepsilon}_{0}$, it follows that
\begin{eqnarray*}
e_{n} = \widetilde{\varepsilon}_{k} &=& e_{n-1} + Q_{k}^{-1} (\widetilde{\varepsilon}_{k-1} - e_{n-1} + \widetilde{\rho}_{k} ) \\
	&=& e_{n-1} + Q_{k}^{-1}Q_{k-1}^{-1}(\widetilde{\varepsilon}_{k-2} - e_{n-1} + \widetilde{\rho}_{k-1}) + Q_{k}^{-1}\widetilde{\rho}_{k} \\
	&\vdots& \\
	&=& e_{n-1} + P^{-1}( \widetilde{\varepsilon}_{0} - e_{n-1} +  \widetilde{\rho}_{1} ) + \sum_{j=2}^{k} Q_{k}^{-1}Q_{k-1}^{-1}\cdots Q_{j}^{-1} \widetilde{\rho}_{j} \\
	&=& e_{n-1} + P^{-1}( -\theta Z_{0} + (\tfrac{1}{2}+\theta)Z )e_{n-1} + P^{-1}(\theta Z_{0} + (\tfrac{1}{2}-\theta)Z)\varepsilon_{k} \\
	&& + \ P^{-1} (\rho_{0} + \widehat{\rho}_{0} + \widetilde{\rho}_{0}) + \sum_{j=1}^{k} Q_{k}^{-1}Q_{k-1}^{-1}\cdots Q_{j}^{-1} \widetilde{\rho}_{j}. 
\end{eqnarray*}
In a similar way, using \eqref{eq:vareps}, it is seen that
\begin{eqnarray*}
\varepsilon_{k} &=& e_{n-1} + P^{-1}( \varepsilon_{0} - e_{n-1} + \rho_{1}) + \sum_{j=2}^{k} Q_{k}^{-1}Q_{k-1}^{-1}\cdots Q_{j}^{-1} \rho_{j} \\
			&=& (I + P^{-1}Z) e_{n-1} + P^{-1}\rho_{0} + \sum_{j=1}^{k} Q_{k}^{-1}Q_{k-1}^{-1}\cdots Q_{j}^{-1} \rho_{j}.
\end{eqnarray*}
Inserting the obtained expression for $\varepsilon_{k}$ into that for $e_{n}$, we arrive at the useful recursion formula
\begin{equation}\label{eq:recursion}
e_{n} = R e_{n-1} + d_{n}
\end{equation}
with {\it stability matrix}
\begin{equation}\label{eq:stabmatrix}
R = I + P^{-1}Z + P^{-1} ( \theta Z_{0} + (\tfrac{1}{2} - \theta) Z ) P^{-1}Z 
\end{equation}
and vector
\begin{eqnarray}\label{eq:localerror1}
d_{n} &=& P^{-1} ( \theta Z_{0} + (\tfrac{1}{2} - \theta) Z )( P^{-1} \rho_{0} + \sum_{j=1}^{k} Q_{k}^{-1} Q_{k-1}^{-1} \cdots Q_{j}^{-1} \rho_{j} ) \nonumber \\
       &&  + \ P^{-1}( \rho_{0} + \widehat{\rho}_{0} + \widetilde{\rho}_{0} ) + \sum_{j=1}^{k} Q_{k}^{-1} Q_{k-1}^{-1} \cdots Q_{j}^{-1} \widetilde{\rho}_{j}. 
\end{eqnarray} 
The recursion \eqref{eq:recursion} implies
\begin{equation}\label{eq:recursion2}
e_{n} = R^{n}e_{0} + \sum_{j=1}^{n} R^{n-j} d_{j}.
\end{equation}

\subsection{Local discretization errors  }\label{sec:LocalError}
We now consider the perturbed MCS scheme \eqref{eq:MCSPerturbed} where the perturbations are such that
$$
U_{n-1}^{*} = U(t_{n-1}),\quad
Y_{j}^{*}=\widetilde{Y}_{j}^{*} = U(t_{n})\quad (0 \leq j \leq k)\quad {\rm and}\quad
\widehat{Y}_{0}^{*} = U(t_{n}).
$$ 
With this choice, $d_n$ is the {\it local discretization error} and $e_n = U(t_{n}) - U_{n}$ the {\it global discretization error} in the $n$-th step.

For the convergence analysis of any given time stepping scheme applied to semidiscrete PDEs to be practical, it is imperative that the pertinent stability and error bounds are not adversely affected by the (arbitrarily small) spatial mesh width employed in the semidiscretization.
Accordingly, {\it by the notation $\mathcal{O}\left((\Delta t)^p\right)$ we shall always mean that the norm $\Vert \cdot \Vert_{2}$ of the term under consideration is bounded by a positive constant times $(\Delta t)^p$ where the constant is independent of the spatial mesh width, the time step size $\Delta t>0$ and the step number $n\ge 1$ with $n \Delta t \le T$. 
If $p=0$, then we write $\mathcal{O}\left(1\right)$ for short.}

Throughout this paper we will assume that the MCS scheme is {\it stable} in the sense that there exists a constant $M$ such that the stability matrix satisfies $\Vert R^{n} \Vert_{2} \leq M$ uniformly in the spatial mesh width, $\Delta t >0$ and integer $n\ge 1$.
Thus, $R^n = \mathcal{O}\left(1\right)$.

To arrive at an optimal convergence order $p$, it turns out that a careful investigation of the local discretization errors $d_n$ is required.
Define 
\[
\varphi_{j}(t)=F_{j}(t,U(t)) \quad {\rm for} \quad 0\le t\le T,~ 0 \le j \le k. 
\]
We assume that the vector functions $\varphi_{j}$ are twice continuously differentiable and that their second derivatives are bounded on $[0,T]$ uniformly in the spatial mesh width. 
Notice that $U'(t) = \sum_{j=0}^k \varphi_{j}(t)$, so the above smoothness condition for the $\varphi_{j}$ implies one for $U$ too.
By Taylor expansion it directly follows that
\begin{subeqnarray}
\label{eq:TruncationErrors}
 \rho_{0} &=& U''(t_{n-1})\ \tfrac{1}{2}(\Delta t)^{2}  + \mathcal{O}\left((\Delta t)^{3}\right), \\
 \widehat{\rho}_{0} &=& - \varphi_{0}'(t_{n-1})\ \theta (\Delta t)^{2} + \mathcal{O}\left((\Delta t)^{3}\right), \\
 \widetilde{\rho}_{0} &=& - U''(t_{n-1})\ (\tfrac{1}{2} - \theta) (\Delta t)^{2}  + \mathcal{O}\left((\Delta t)^{3}\right), \\
 \rho_{j} &=& \widetilde{\rho}_{j} ~~=~ - \varphi_{j}'(t_{n-1})\ \theta (\Delta t)^{2} +\mathcal{O}\left((\Delta t)^{3}\right), \qquad 1 \le j \le k.
\end{subeqnarray}
Since $\rho_{j} =\widetilde{\rho}_{j}$ for all $1 \le j \le k$, the expression \eqref{eq:localerror1} for $d_{n}$ becomes
$$ 
d_{n} = \left( I + P^{-1} ( \theta Z_{0} + (\tfrac{1}{2} - \theta ) Z ) \right) ( P^{-1} \rho_{0} + \sum_{j=1}^{k} Q_{k}^{-1}Q_{k-1}^{-1}\cdots Q_{j}^{-1} \rho_{j} ) + P^{-1}( \widehat{\rho}_{0} + \widetilde{\rho}_{0} ). 
$$
Inserting the expansions \eqref{eq:TruncationErrors} into this and taking into account the uniform boundedness of the matrices $Q_{j}^{-1}$ ($1 \le j \le k$), see \eqref{eq:PQbounded}, we obtain
\begin{eqnarray*}
d_{n} &=& \left( I + P^{-1} ( \theta Z_{0} + (\tfrac{1}{2} - \theta ) Z ) \right) P^{-1} U''(t_{n-1})\ \tfrac{1}{2}(\Delta t)^{2} \nonumber \\ 
		&& - \ \left( I + P^{-1} ( \theta Z_{0} + (\tfrac{1}{2} - \theta ) Z ) \right) \sum_{j=1}^{k} Q_{k}^{-1}Q_{k-1}^{-1}\cdots Q_{j}^{-1} \varphi_{j}'(t_{n-1})\ \theta (\Delta t)^{2} \nonumber \\
		&& - \ P^{-1} \varphi_{0}'(t_{n-1})\ \theta (\Delta t)^{2} - P^{-1} U''(t_{n-1})\ (\tfrac{1}{2}- \theta) (\Delta t)^{2}  +\phantom{\sum_{j=1}^{k}} \nonumber \\
		&& + \ \left( I + P^{-1} ( \theta Z_{0} + (\tfrac{1}{2} - \theta ) Z ) \right) \mathcal{O}\left((\Delta t)^{3}\right) +\mathcal{O}\left((\Delta t)^{3}\right). 
\end{eqnarray*}
Using that $U''(t) = \sum_{j=0}^{k} \varphi_{j}'(t) $ the latter expression for $d_n$ can be written in the following form, which will be employed in the next section,
\begin{eqnarray}
\label{eq:DiscretizationErrors}
d_{n} &=& P^{-1} ( \theta Z_{0} + (\tfrac{1}{2} - \theta) Z ) P^{-1} ( \tfrac{1}{2}U''(t_{n-1}) - \theta \sum_{j=1}^{k} \varphi_{j}'(t_{n-1}) )\ (\Delta t)^{2} \nonumber \\
&& + \ \left( I + P^{-1} ( \theta Z_{0} + (\tfrac{1}{2} - \theta) Z ) \right) P^{-1} \sum_{j=2}^{k} (I - Q_{1}Q_{2}\cdots Q_{j-1} ) \varphi_{j}'(t_{n-1})\ \theta (\Delta t)^{2}  \nonumber \\
&& + \ \left( I + P^{-1} ( \theta Z_{0} + (\tfrac{1}{2} - \theta ) Z ) \right) \mathcal{O}\left((\Delta t)^{3}\right) +\mathcal{O}\left((\Delta t)^{3}\right). 
\end{eqnarray} 

\subsection{Main convergence theorem  }
From \eqref{eq:DiscretizationErrors} and the specific dependence of the matrices $Z_0$, $Z$, $Q_j$ ($1\le j\le k$) on $\Delta t$ it is readily seen that for any given {\it fixed}\, semidiscrete system the local errors are bounded by a constant times $(\Delta t)^{3}$. 
Next, formula \eqref{eq:recursion2} together with the stability of the MCS scheme directly imply a well-known estimate for the global discretization errors in terms of the local discretization errors (note that $e_0=0$),
\begin{equation*}
\label{eq:GlobalErrorBound}
\Vert e_n \Vert_{2} \le M \sum_{j=1}^{n} \Vert d_j \Vert_{2}.
\end{equation*}
Hence, it follows that the global errors are bounded by a constant times $(\Delta t)^{2}$, that is second-order convergence.
However, as the spatial mesh width decreases (and the size $m$ of the semidiscrete system increases), the pertinent error constant can become arbitrarily large due to negative powers of the spatial mesh width occurring in the matrices $A_j$ ($0\le j\le k$).
Clearly, this renders the global error bound obtained in this way impractical.

In the following we shall present for the MCS scheme with $k=2$ a useful second-order convergence result, which is valid uniformly in the spatial mesh width.
We apply a key lemma from Hundsdorfer \cite{Hun92}, cf.~also \cite{Hun02,HV03}.
For completeness, its (short) proof is included. 
\begin{lemma}
\label{lemma:Hundsdorfer}
Let $\alpha >0$.
If the time stepping scheme is stable and the local discretization errors satisfy 
\begin{equation}\label{eq:LocalErrorDecomp}
d_{j} = (R-I)\xi_{j} + \eta_{j} \quad {\rm with} \quad \xi_{j} = \mathcal{O}\left((\Delta t)^{\alpha} \right),~
\xi_{j} - \xi_{j-1}= \mathcal{O}\left((\Delta t)^{\alpha +1} \right),~ \eta_{j} = \mathcal{O}\left((\Delta t)^{\alpha +1} \right), 
\end{equation}
then for the global discretization errors one has that $e_{n} = \mathcal{O}\left((\Delta t)^{\alpha} \right)$.
\end{lemma}
\begin{proof}
Consider relation \eqref{eq:recursion2} for the global error.
Inserting $d_{j} = (R-I)\xi_{j} + \eta_{j}$ and $e_0=0$ gives
\[
e_{n} = R^{n}\xi_{1} - \xi_{n} + \sum_{j=2}^{n} R^{n-j+1} (\xi_{j} - \xi_{j-1}) + \sum_{j=1}^{n} R^{n-j} \eta_{j}. 
\]
By stability of the time stepping scheme (cf.~Section \ref{sec:LocalError}), this leads to the bound
\[
\Vert e_n \Vert_{2} \le M \Vert \xi_{1} \Vert_{2} + \Vert \xi_{n} \Vert_{2} + M \sum_{j=2}^{n} \Vert \xi_{j} - \xi_{j-1} \Vert_{2} + M \sum_{j=1}^{n}  \Vert \eta_{j} \Vert_{2}.
\]
Using the properties of $\xi_{j}$ and $\eta_{j}$ in \eqref{eq:LocalErrorDecomp} and $n \Delta t \leq T$, the assertion of the lemma follows.
\begin{flushright}
\mbox{\tiny $\blacksquare$}
\end{flushright}
\end{proof}

If $k=2$, then the obtained expression \eqref{eq:DiscretizationErrors} for the local error simplifies to $d_{n} = d_{n}^{(1)} + d_{n}^{(2)} + d_{n}^{(3)}$ with
\begin{equation*}
\label{eq:LokaleFoutOpgesplitst}
\begin{array}{l}
d_{n}^{(1)} = P^{-1} (\theta Z_{0} + (\tfrac{1}{2} - \theta) Z ) P^{-1} ( \tfrac{1}{2} U''(t_{n-1}) - \theta \sum_{j=1}^{2} \varphi_{j}'(t_{n-1}) )\ (\Delta t)^{2}, \\\\
d_{n}^{(2)} = \left(I + P^{-1} (\theta Z_{0} + (\tfrac{1}{2} - \theta) Z ) \right) P^{-1} Z_{1}\, \varphi_{2}'(t_{n-1})\ \theta^{2} (\Delta t)^{2}, \\\\
d_{n}^{(3)} = \left(I + P^{-1} (\theta Z_{0} + (\tfrac{1}{2} - \theta) Z ) \right) \mathcal{O}\left( (\Delta t)^{3} \right) + \mathcal{O}\left( (\Delta t)^{3} \right).
\end{array}
\end{equation*}
Using formula \eqref{eq:stabmatrix} for the stability matrix, the first two components of $d_{n}$ can be rewritten as (assuming the pertinent inverses exist),
\begin{equation*}
\label{eq:Decomposition2}
\begin{array}{llr}
d_{n}^{(1)} &= (R - I) Z^{-1} P (I + P^{-1} (\theta Z_{0} + (\tfrac{1}{2}-\theta) Z ))^{-1} d_{n}^{(1)} \\\\
			&= (R - I) Z^{-1} (\theta Z_{0} + (\tfrac{1}{2}-\theta) Z )(P+\theta Z_{0} + (\tfrac{1}{2} -\theta )Z )^{-1} ( \tfrac{1}{2} U''(t_{n-1}) - \theta \sum_{j=1}^{2} \varphi_{j}'(t_{n-1}) )\ (\Delta t)^{2}, \\\\
d_{n}^{(2)} &= (R - I)  Z^{-1}Z_{1} \varphi_{2}'(t_{n-1})\ \theta^{2} (\Delta t)^{2}.
\end{array}
\end{equation*}
Upon invoking Lemma \ref{lemma:Hundsdorfer} with $\alpha=2$, we then arrive at the main result of this paper.
\begin{theorem}
\label{th:MCSConvergentOrder2}
Let $k=2$.
Assume that the $\varphi_{j}$ ($j=0,1,2$) are twice continuously differentiable and their second derivatives are bounded on $[0,T]$ uniformly in the spatial mesh width. 
Assume $(A_{j}v,v) \leq 0 $ whenever $v \in \mathbb{R}^{m}$ and $j=1,2$.
Assume the MCS scheme is stable, the matrices $A$ and $P + \theta Z_{0} + (\tfrac{1}{2} - \theta)Z$ are invertible and the matrices 
\begin{equation}
\label{eq:fourmatrices}
A^{-1}A_{1},~ A^{-1}A_{2},~ I + P^{-1} ( \theta Z_{0} + (\tfrac{1}{2} - \theta) Z),~ (P + \theta Z_{0} + (\tfrac{1}{2} - \theta)Z)^{-1}
\end{equation}
are all $\mathcal{O}\left(1 \right)$. 
Then the global discretization errors for the MCS scheme satisfy $e_{n} = \mathcal{O}\left( (\Delta t)^{2} \right)$.
\end{theorem}

\subsection{Boundedness assumptions in Theorem \ref{th:MCSConvergentOrder2}  }
\label{subsec:AssumptionsBoundedness}
The assumptions concerning the boundedness of the matrices \eqref{eq:fourmatrices} in Theorem \ref{th:MCSConvergentOrder2} are similar to those made in \cite{Hun02} in order to prove convergence of the HV scheme with $k=1$. 
The uniform boundedness of $A^{-1}A_{1}$ and $A^{-1}A_{2}$ was also assumed there and is often fulfilled. 
If $k=1$, the assumption $(P + \theta Z_{0} + (\tfrac{1}{2} - \theta)Z)^{-1} = \mathcal{O}\left(1 \right)$ is closely related to the condition $(29)$ in \cite{Hun02} for the HV scheme. 
The assumption $I + P^{-1} ( \theta Z_{0} + (\tfrac{1}{2} - \theta) Z) = \mathcal{O}\left(1 \right)$ can be viewed as a counterpart of the condition $I - P^{-1} + \tfrac{1}{2} P^{-1}Z  = \mathcal{O}\left(1 \right)$ which was tacitly assumed in \cite{Hun02}.

For $k=2$ the conditions $I + P^{-1} ( \theta Z_{0} + (\tfrac{1}{2} - \theta) Z)= \mathcal{O}\left(1 \right)$ and $(P + \theta Z_{0} + (\tfrac{1}{2} - \theta)Z)^{-1} = \mathcal{O}\left(1 \right)$ are new in the literature. 
To gain insight into these conditions, we follow the well-known von Neumann framework and consider the two-dimensional convection-diffusion equation 
\begin{equation} 
\label{eq:ModelEquation}
u_{t} = d_{11}u_{xx} + 2d_{12}u_{xy} + d_{22}u_{yy} + c_{1}u_{x} + c_{2} u_{y}
\end{equation}
for $(x,y) \in (0,1) \times (0,1)$ and $0 \leq t \leq T$ with periodic boundary condition. 
Here $ c_{1}, c_{2}, d_{11}, d_{12}, d_{22}$ denote given real constants with 
\begin{equation}
\label{eq:Coefficients} 
d_{11} \geq 0, \quad d_{22} \geq 0, \quad \vert d_{12} \vert \leq \gamma \sqrt{ d_{11}d_{22} }\quad {\rm and}\quad \gamma \in [0,1].
\end{equation}
After semidiscretization of \eqref{eq:ModelEquation} by standard finite difference schemes on uniform rectangular grids, the analysis reduces to bounding from above the two scalar terms
\begin{equation}
\label{eq:scalarversions}
\vert 1 + \tfrac{1}{2}\tfrac{z_{0}}{p} + (\tfrac{1}{2}-\theta)\tfrac{z_{1}+z_{2}}{p} \vert \quad \mbox{and} \quad \vert p + \tfrac{1}{2} z_{0} + (\tfrac{1}{2}-\theta)(z_{1} + z_{2}) \vert^{-1}
\end{equation} 
with $p = (1 - \theta z_{1}) (1- \theta z_{2} )$ for all complex numbers $z_{0},z_{1},z_{2}$ satisfying
\begin{equation}
\label{eq:scalars}
\mathcal{R}z_{1}\leq 0,~~ \ \mathcal{R}z_{2} \leq 0,~~ \ \vert z_{0} \vert \leq 2 \gamma \sqrt{\mathcal{R}z_{1}\mathcal{R}z_{2}}.
\end{equation}
The condition (\ref{eq:scalars}) arises naturally in the von Neumann stability analysis of ADI schemes when a mixed derivative $u_{xy}$ is present.
It has been considered in \cite{IHM11,IHW07,IHW09} with $\gamma = 1$ and in \cite{IHM13, M14} for arbitrary $\gamma \in [0,1]$.

For the first term in \eqref{eq:scalarversions}, we obtain the following positive result under (\ref{eq:scalars}).
\begin{theorem}
\label{th:StabilityType}
Assume (\ref{eq:scalars}) and $0\le \gamma \le 1$. Then 
\[ 
\left\vert  1 + \tfrac{1}{2} \frac{z_{0}}{p} + (\tfrac{1}{2}-\theta)\frac{z_{1}+z_{2}}{p} \right\vert \le
\begin{cases} 
\tfrac{1}{2\theta} - \tfrac{3}{2} &\mbox{if}~~ 0 < \theta < \tfrac{1}{6}\,, \\ 
\tfrac{3}{2} & \mbox{if}~~ \tfrac{1}{6} \le \theta \le 1\,,\\
2 - \tfrac{1}{2\theta} & \mbox{if}~~ 1 < \theta\,.\\
\end{cases}
\]
\end{theorem}
\begin{proof}
By \cite[Lemma 2.3]{IHW07} it holds that
\[
p \not= 0 \quad {\rm and} \quad
\left\vert \alpha \right\vert + \left\vert \beta \right\vert \le \frac{1}{2\theta} \quad {\rm with} \quad
\alpha = \frac{z_{0}}{p}~,~ \beta = \frac{1}{2\theta} + \frac{z_{1}+z_{2}}{p}\,.
\]
Using this we obtain
\begin{eqnarray*}
\left\vert  1 + \tfrac{1}{2} \frac{z_{0}}{p} + (\tfrac{1}{2}-\theta)\frac{z_{1}+z_{2}}{p} \right\vert
&=&\left\vert 1 + \tfrac{1}{2} \alpha + (\tfrac{1}{2}-\theta)(\beta-\tfrac{1}{2\theta}) \right\vert\\\\
&=& \left\vert \tfrac{3}{2} - \tfrac{1}{4\theta} + \tfrac{1}{2} \alpha + (\tfrac{1}{2}-\theta) \beta \right\vert\\\\
&\le& \left\vert \tfrac{3}{2} - \tfrac{1}{4\theta}\right\vert + \tfrac{1}{2} \left\vert\alpha\right\vert + \left\vert\tfrac{1}{2}-\theta\right\vert \left\vert\beta \right\vert\\\\
&\le& \left\vert \tfrac{3}{2} - \tfrac{1}{4\theta}\right\vert + \tfrac{1}{2\theta}\cdot\max \{\tfrac{1}{2},\left\vert\tfrac{1}{2}-\theta\right\vert \}\,,
\end{eqnarray*}
which readily yields the result of the theorem.
\begin{flushright}
\mbox{\tiny $\blacksquare$}
\end{flushright}
\end{proof}
For bounding the second term in \eqref{eq:scalarversions}, we make use of the following elementary lemma, where $\R^{+} = [0,\infty)$. 
\begin{lemma}
\label{le:RealFunction}
Let $0 \leq \delta \leq 1$ and 
$$ 
f: \R^{+} \times \R^{+} \rightarrow \R : (x,y) \rightarrow \sqrt{1+ x^{2}}\sqrt{1+ y^{2}} - \delta(x + y). 
$$
Then
$$ 
\min\{ f(x,y)\, \vert\, (x,y) \in \R^{+} \times \R^{+} \} = 1 - \delta^{2}. 
$$
\end{lemma}
The proof of Lemma \ref{le:RealFunction} is given in Appendix \ref{ap:prooflemma}.
\begin{theorem}
\label{th:begrensdeinverse}
Assume (\ref{eq:scalars}) and $0\le \gamma \le 1$. Then
$$
\vert p + \tfrac{1}{2}z_{0}+ (\tfrac{1}{2} - \theta )(z_{1} + z_{2}) \vert \geq 
\begin{cases} 
(\theta - \tfrac{1}{4})/\theta^{2} & \mbox{if}~~ \tfrac{1}{4} \leq \theta < \tfrac{1}{2}, \ 0 \leq \gamma < 2\theta \,, \\
-3(\theta - \tfrac{1+\gamma}{6})(\theta -\tfrac{1+\gamma}{2})/\theta^2 &\mbox{if}~~ \tfrac{1}{4} \leq \theta < \tfrac{1}{2}, \ 2\theta \leq \gamma \leq \min\{6\theta - 1,1\} \,, \\ 
1 & \mbox{if}~~ \tfrac{1}{2} \leq \theta\,.\\
\end{cases} 
$$
\end{theorem}
\begin{proof}
First, consider the case $\theta \geq \tfrac{1}{2}$ and put $\digamma = 2\theta - \tfrac{1}{2} $. Then,
\begin{eqnarray*}
\vert p+ \tfrac{1}{2}z_{0} + (\tfrac{1}{2} - \theta )(z_{1} + z_{2}) \vert &\geq &   \vert p + (\tfrac{1}{2} - \theta )(z_{1} + z_{2}) \vert - \sqrt{\mathcal{R}z_{1}\mathcal{R}z_{2}} \\
& = & \vert (\tfrac{\digamma}{\theta} - \theta z_{1} ) (\tfrac{\digamma}{\theta} - \theta z_{2}) + 1 - \tfrac{\digamma^{2}}{\theta^{2}} \vert - \sqrt{\mathcal{R}z_{1}\mathcal{R}z_{2}} \\
& \geq & \vert \tfrac{\digamma}{\theta} - \theta z_{1} \vert \vert \tfrac{\digamma}{\theta} - \theta z_{2} \vert - \vert 1 - \tfrac{\digamma^{2}}{\theta^{2}} \vert - \sqrt{\mathcal{R}z_{1}\mathcal{R}z_{2}}.
\end{eqnarray*}
Now, since 
$$\theta^{2} - \digamma^{2} = -3\theta^{2} +2\theta - \tfrac{1}{4} = -3(\theta - \tfrac{1}{6})(\theta - \tfrac{1}{2}),$$ it holds that $1 - \tfrac{\digamma^{2}}{\theta^{2}} $ is negative.
Further, since $\mathcal{R}(\tfrac{\digamma}{\theta} - \theta z_{j}) \geq 0$,
we have that
$$ \vert \tfrac{\digamma}{\theta} - \theta z_{j} \vert \geq \tfrac{\digamma}{\theta} - \theta \mathcal{R}z_{j} \quad \mbox{for} \ j=1,2.$$
As a consequence
\begin{eqnarray*}
\vert p+ \tfrac{1}{2}z_{0} + (\tfrac{1}{2} - \theta )(z_{1} + z_{2}) \vert &\geq & \tfrac{\digamma^{2}}{\theta^{2}} - \digamma(\mathcal{R}z_{1}+\mathcal{R}z_{2}) + \theta^{2}\mathcal{R}z_{1}\mathcal{R}z_{2} +1 - \tfrac{\digamma^{2}}{\theta^{2}} - \sqrt{\mathcal{R}z_{1}\mathcal{R}z_{2}} \\
& = & 1 + \theta^{2}\mathcal{R}z_{1}\mathcal{R}z_{2} + \digamma ( \sqrt{-\mathcal{R}z_{1}} - \sqrt{-\mathcal{R}z_{2}}\, )^{2} + (2\digamma-1)\sqrt{\mathcal{R}z_{1}\mathcal{R}z_{2}} \\
& \geq & 1 + 2(2\theta -1)\sqrt{\mathcal{R}z_{1}\mathcal{R}z_{2}} \\
&\geq & 1,
\end{eqnarray*}
which completes the proof for $\theta \geq \tfrac{1}{2}$. 

Next consider the case $\tfrac{1}{4} \leq \theta < \tfrac{1}{2}$ and $2\theta \leq \gamma \leq \min\{6\theta-1,1\}$.
Define vectors 
$$ 
v_j =
\left[ \begin{array}{c}
\sqrt{-2\theta \mathcal{R}z_{j}} \\\\
\sqrt{1+\theta^{2} \vert z_{j} \vert^{2}} 
\end{array} \right] \quad \mbox{  for} \ j=1,2.
$$
By the Cauchy--Schwarz inequality, we have
\begin{equation}
\label{eq:CSinequality}
\vert p \vert = 
\sqrt{v_1^{\rm T} v_1} \sqrt{v_2^{\rm T} v_2} \ge v_1^{\rm T} v_2 =
2\theta \sqrt{\mathcal{R}z_{1}\mathcal{R}z_{2}} + \sqrt{1+\theta^{2} \vert z_{1} \vert^{2}}\sqrt{1+\theta^{2} \vert z_{2} \vert^{2}}. 
\end{equation}
Also, 
\begin{eqnarray}
\vert z_{1} + z_{2} \vert &=& \sqrt{(\mathcal{R}z_{1} + \mathcal{R}z_{2})^{2} + (\mathcal{I}z_{1} + \mathcal{I}z_{2})^{2}} \nonumber \\
& =& \sqrt{(\mathcal{R}z_{1} - \mathcal{R}z_{2})^{2} + 4\mathcal{R}z_{1}\mathcal{R}z_{2} + (\mathcal{I}z_{1} + \mathcal{I}z_{2})^{2}} \nonumber \\
& \geq & 2 \sqrt{\mathcal{R}z_{1}\mathcal{R}z_{2}} \nonumber \\
& \geq & \frac{1}{\gamma}\vert z_{0}\vert.
\label{eq:z0inequality}
\end{eqnarray}
Using (\ref{eq:CSinequality}), (\ref{eq:z0inequality}) and $\tfrac{1}{2} - \tfrac{\theta}{\gamma}\ge 0$
we find 
\begin{eqnarray*}
\vert p+ \tfrac{1}{2}z_{0} + (\tfrac{1}{2} - \theta )(z_{1} + z_{2}) \vert 
& \geq & \vert p \vert - \tfrac{1}{2} \vert z_{0} \vert - (\tfrac{1}{2} - \theta )\vert z_{1} + z_{2} \vert\\
& \geq & \vert p \vert - 2\theta\sqrt{\mathcal{R}z_{1}\mathcal{R}z_{2}} - (\tfrac{1}{2} - \tfrac{\theta}{\gamma}) \vert z_{0} \vert - (\tfrac{1}{2} - \theta) \vert z_{1} + z_{2} \vert  \\
& \geq & \sqrt{1+\theta^{2} \vert z_{1} \vert^{2}}\sqrt{1+\theta^{2} \vert z_{2} \vert^{2}} - (\tfrac{1+\gamma}{2} - 2\theta) \vert z_{1} + z_{2} \vert \\
& \geq & \sqrt{1+\theta^{2} \vert z_{1} \vert^{2}}\sqrt{1+\theta^{2} \vert z_{2} \vert^{2}} - (\tfrac{1+\gamma}{2} - 2\theta) (\vert z_{1} \vert + \vert z_{2} \vert).
\end{eqnarray*}
Define 
$$
\delta = (\tfrac{1+\gamma}{2} - 2\theta)/ \theta.
$$ 
It is easily verified that, in the case under consideration,  $0<\delta \leq 1$ and application of Lemma \ref{le:RealFunction} yields
\begin{eqnarray*}
\vert p+ \tfrac{1}{2}z_{0} + (\tfrac{1}{2} - \theta )(z_{1} + z_{2}) \vert & \geq & 1 - (\tfrac{1+\gamma}{2} - 2\theta)^{2}/\theta^{2} \\
& = & -3(\theta - \tfrac{1+\gamma}{6})(\theta -\tfrac{1+\gamma}{2})/\theta^2.
\end{eqnarray*}
Finally consider the case $\tfrac{1}{4} \leq \theta < \tfrac{1}{2}$ and $\gamma < 2\theta$. 
Analogously as above one finds
\begin{eqnarray*}
\vert p+ \tfrac{1}{2}z_{0} + (\tfrac{1}{2} - \theta )(z_{1} + z_{2}) \vert 
& \geq &  \vert p \vert - 2\theta\sqrt{\mathcal{R}z_{1}\mathcal{R}z_{2}} - (\tfrac{1}{2} - \tfrac{\theta}{\gamma}) \vert z_{0} \vert - (\tfrac{1}{2} - \theta) \vert z_{1} + z_{2} \vert  \\
& \geq & \sqrt{1+\theta^{2} \vert z_{1} \vert^{2}}\sqrt{1+\theta^{2} \vert z_{2} \vert^{2}} - (\tfrac{1}{2} - \theta) \vert z_{1} + z_{2} \vert  \\
& \geq & \sqrt{1+\theta^{2} \vert z_{1} \vert^{2}}\sqrt{1+\theta^{2} \vert z_{2} \vert^{2}} - (\tfrac{1}{2} - \theta) (\vert z_{1} \vert + \vert z_{2} \vert).
\end{eqnarray*}
Applying Lemma \ref{le:RealFunction} with $\delta = (\tfrac{1}{2} - \theta)/\theta$, it then follows that
\begin{eqnarray*}
\vert p+ \tfrac{1}{2}z_{0} + (\tfrac{1}{2} - \theta )(z_{1} + z_{2}) \vert & \geq & 1 - (\tfrac{1}{2} - \theta)^{2}/\theta^{2} \\
& = & (\theta - \tfrac{1}{4})/\theta^{2},
\end{eqnarray*}
and this completes the proof.
\begin{flushright}
\mbox{\tiny $\blacksquare$}
\end{flushright} 
\end{proof}
Theorem \ref{th:begrensdeinverse} directly implies the positive result that the second term in \eqref{eq:scalarversions} is also bounded from above whenever 
$\{ \tfrac{1}{4} < \theta \leq \tfrac{1}{3}$ and  $0 \leq \gamma < 6\theta - 1 \}$ or $\{\theta > \tfrac{1}{3} \}$.

\setcounter{equation}{0}
\setcounter{theorem}{0}
\section{Numerical experiments}
We present numerical experiments in the case of the 2D convection-diffusion equation (\ref{eq:ModelEquation}) for $(x,y) \in \Omega = (0,1) \times (0,1)$ and $0 \leq t \leq 2$ with parameters
\begin{equation}
\label{eq:ParameterSet}
 d_{11} = d, \ d_{12} = -2 \gamma d,\ d_{22} = 4d, \ c_{1} = -2, \ c_{2} = -3 \quad {\rm and}\quad d=0.025,\ \gamma = 0.7. 
\end{equation}
The requirement (\ref{eq:Coefficients}) is fulfilled for this choice of parameters.
We consider the initial condition
\begin{equation*}
u(x,y,0) = e^{-4(\sin^{2}\! \pi x + \cos^{2}\! \pi y)} \qquad \mbox{for}~~ (x,y) \in \Omega,
\end{equation*}
and Dirichlet boundary condition 
\begin{equation*}
~~~~~~~~~~~u(x,y,t) = e^{-rt}u(x,y,0) \qquad \mbox{for}~~ (x,y) \in \partial \Omega, \ 0 < t \leq 2,
\end{equation*}
with $r = 0.05$. 
Semidiscretization of the initial-boundary value problem is performed using standard second-order central finite difference schemes on a rectangular grid in $\Omega$ with spatial mesh widths $\Delta x = 1/(m_{1}+1)$ and $\Delta y = 1/(m_{2}+1)$. 
In order to avoid spurious oscillations, the convection terms $u_x$ and $u_y$ are discretized by second-order backward finite differencing near $x=1$ and $y=1$, respectively.
The semidiscretization leads to an initial value problem (\ref{eq:ODE}) with $m=m_1m_2$ and $F(t,v) = Av + e^{-rt} g$ with given $m\times m$-matrix $A$ and $m$-vector $g$.
Figure \ref{fig:InitialEnd} shows the semidiscrete solutions $U(0)$ and $U(2)$ on the grid in $\Omega$ if $m_1=m_2=50$.

We employ the splitting (\ref{eq:Decomposition}) of the function $F$ with $k=2$ and consider application of the MCS scheme with four interesting parameter values, namely $\theta = \tfrac{1}{4}, \tfrac{1}{3}, \tfrac{1}{2}, 1$. 
Recall that for $\theta = \tfrac{1}{2}$ one recovers the CS scheme \cite{CS88}.
Stability of the MCS scheme pertinent to 2D convection-diffusion equations with mixed derivative term has been analyzed in \cite{IHM11,M14}.
Applying the results from these references to the situation at hand, we expect unconditional stability whenever $\theta \ge \tfrac{1}{3}$ and a lack thereof when $\theta = \tfrac{1}{4}$.

Figure \ref{fig:Exp} displays for $m_{1}=m_{2} \in \{ 50, 100, 150, 200 \}$ the norms of the global discretization errors at $t=2$ as a function of $N$,
\begin{equation*}
e(N,m_{1},m_{2}) = \Vert U(2) - U_{2N} \Vert_{2},
\end{equation*}
where $N \Delta t = 1$ and $\| \cdot \|_2$ denotes the scaled Euclidean vector norm from Section~\ref{intro}.
Here we applied the HV scheme \cite{IHW07,Hun02,V99} with $\theta = \tfrac{1}{2} + \tfrac{1}{6}\sqrt{3}$ and $N = 10^4$ to obtain a reference solution~$U(2)$. 

For $\theta = \tfrac{1}{3}, \tfrac{1}{2}, 1$ the results of Figure \ref{fig:Exp} clearly reveal a second-order convergence behaviour in $\Delta t$, uniformly in the spatial mesh width. 
This positive conclusion is in line with our theory of Section~\ref{sec:results}.
For $\theta = \tfrac{1}{4}$, a second-order convergence behaviour uniformly in the spatial mesh width is clearly absent.
This conclusion also agrees with the theory of Section~\ref{sec:results}.
The observed strong increase in the global discretization errors as the spatial mesh width decreases corresponds to a lack of unconditional stability when $\theta = \tfrac{1}{4}$ (cf.~above).

We have performed numerical experiments also for other convection-diffusion parameter sets than (\ref{eq:ParameterSet}) which satisfy the condition (\ref{eq:Coefficients}) as well as for the celebrated two-dimensional Heston model \cite{H93} from financial mathematics. 
Semidiscretization of the latter model was performed as described in \cite{IHF10} on a non-uniform spatial grid, also leading to initial value problems of type (\ref{eq:ODE}) with $F(t,v) = Av + g(t)$.
In all of the experiments, we found the obtained conclusions concerning the temporal convergence behaviour of the MCS scheme to be in line with the theory of Section~\ref{sec:results}.

\begin{figure}
\begin{center}
\includegraphics[scale=0.5]{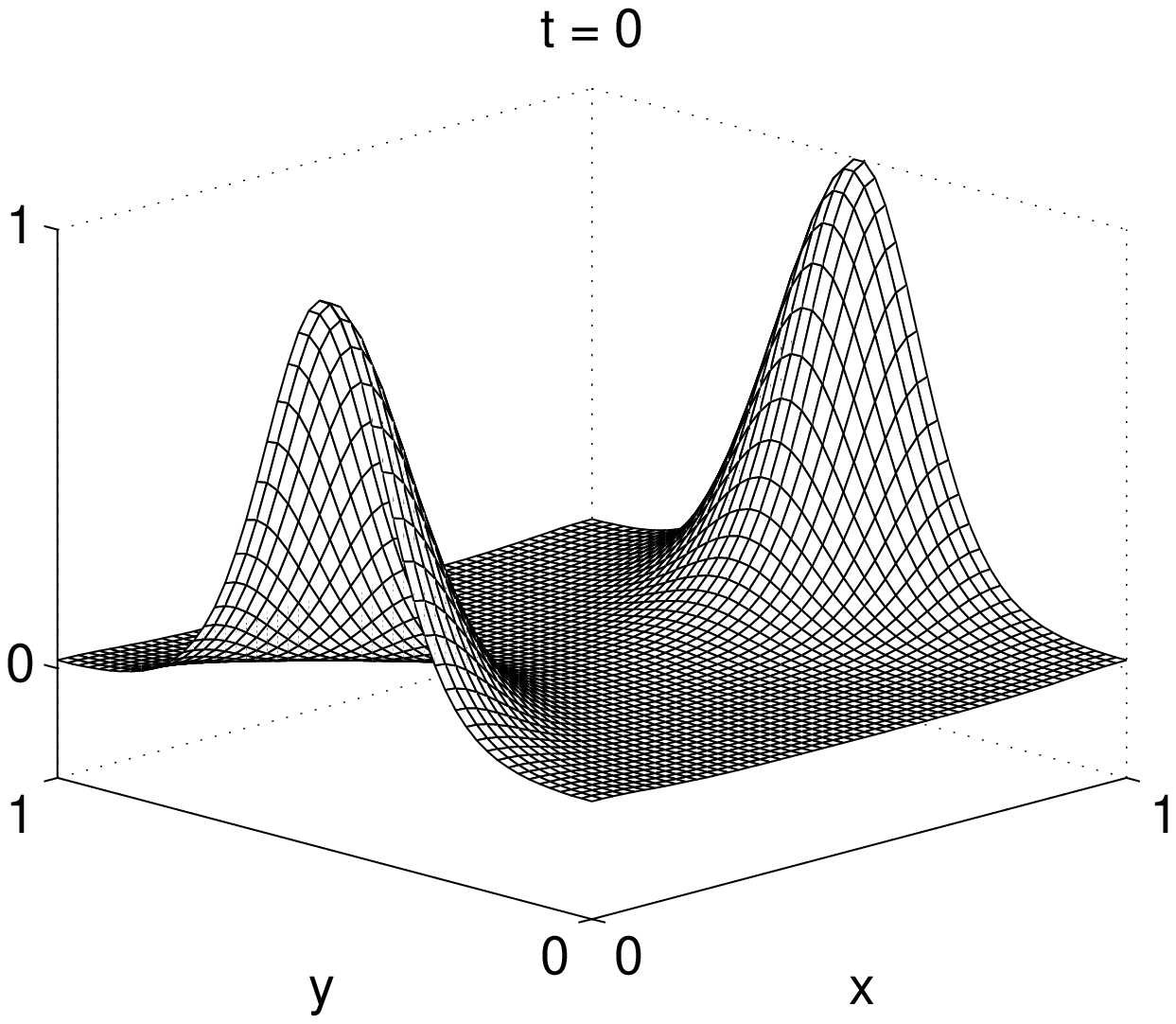}
\includegraphics[scale=0.5]{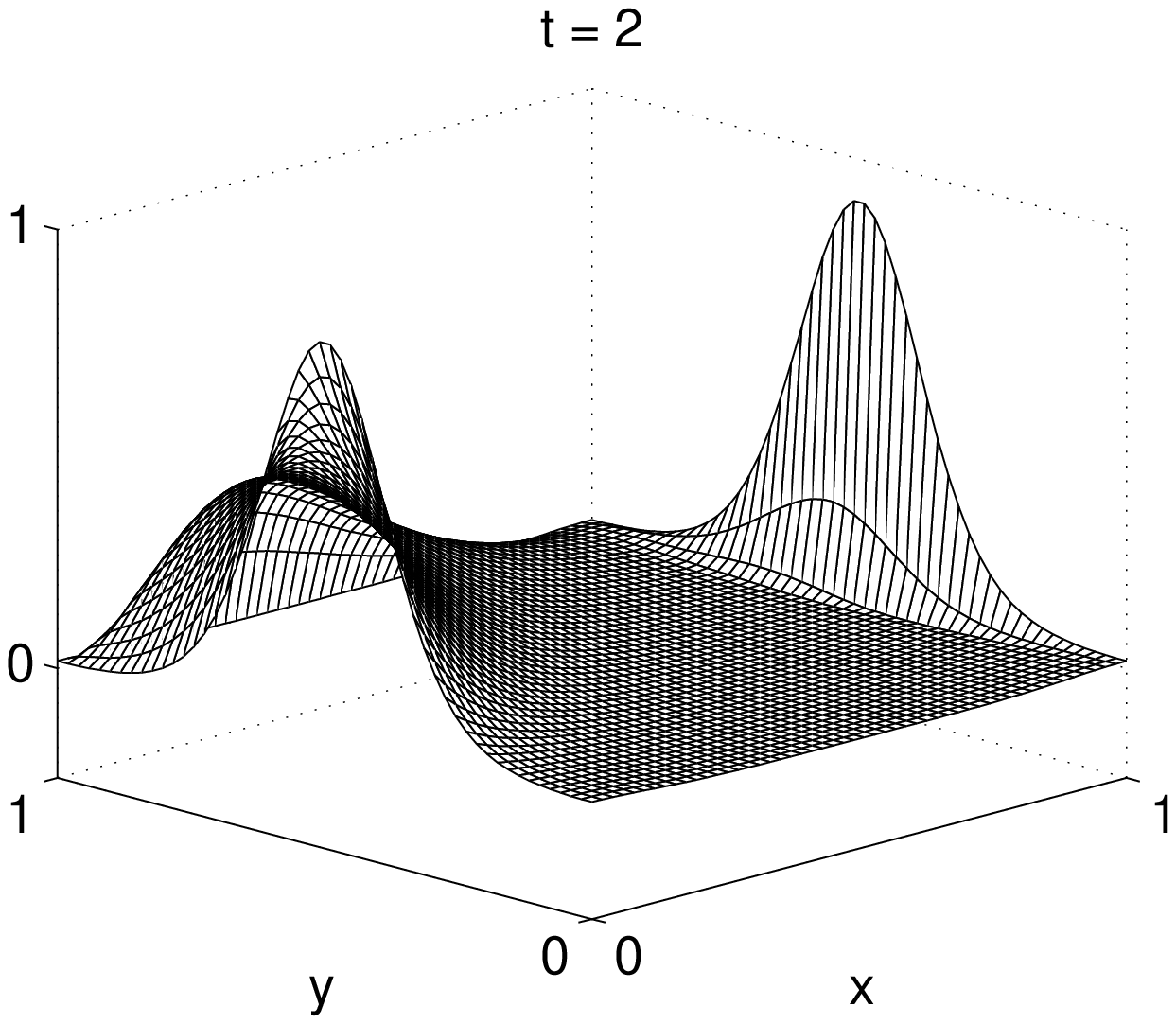}
\caption{Semidiscrete solutions $U(t)$ on $\Omega$ for $t=0, 2$ if $m_1=m_2=50$.}
\label{fig:InitialEnd}
\end{center}
\end{figure}

\begin{figure}
\begin{center}
\includegraphics[scale=0.5]{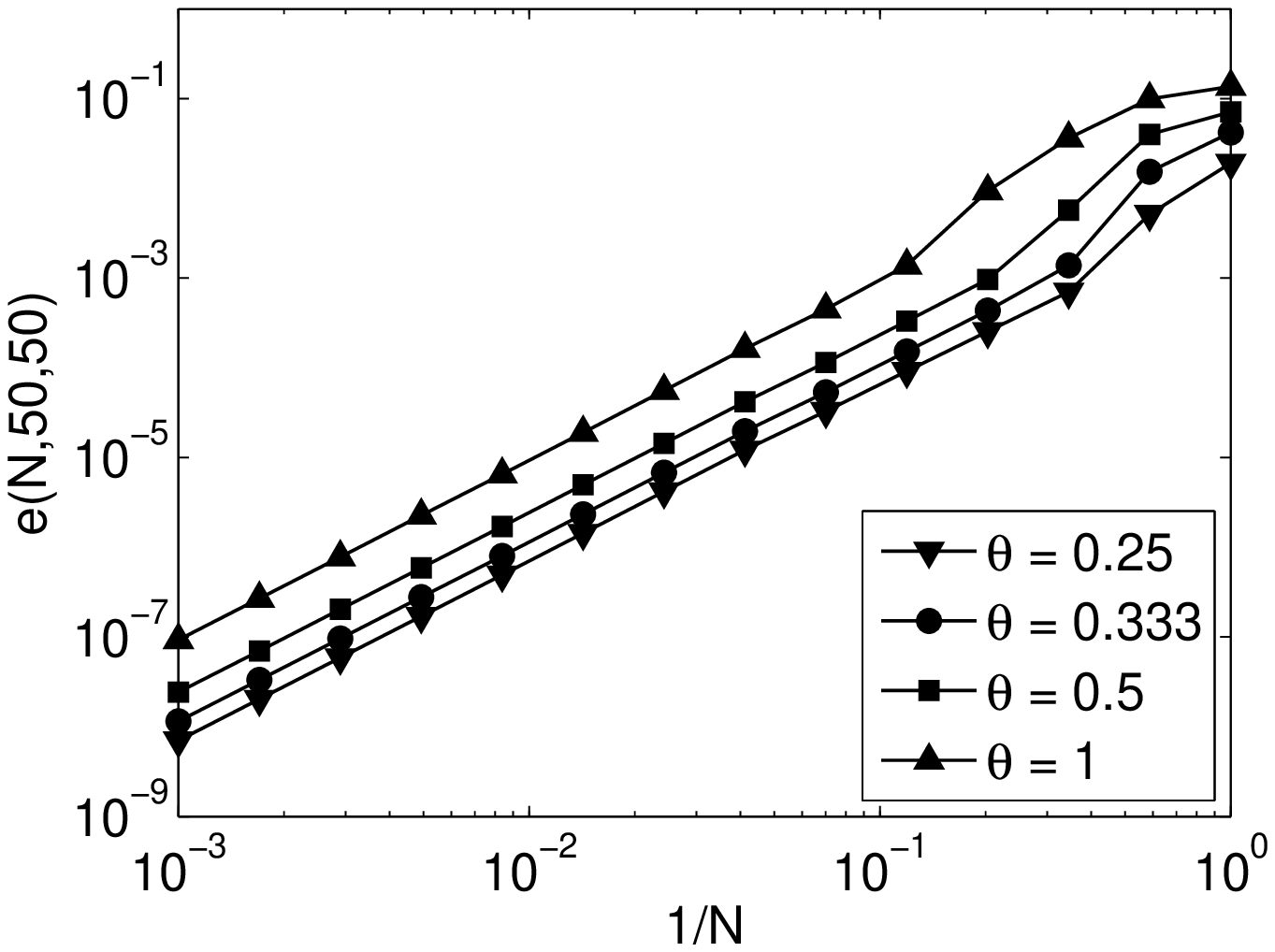}
\includegraphics[scale=0.5]{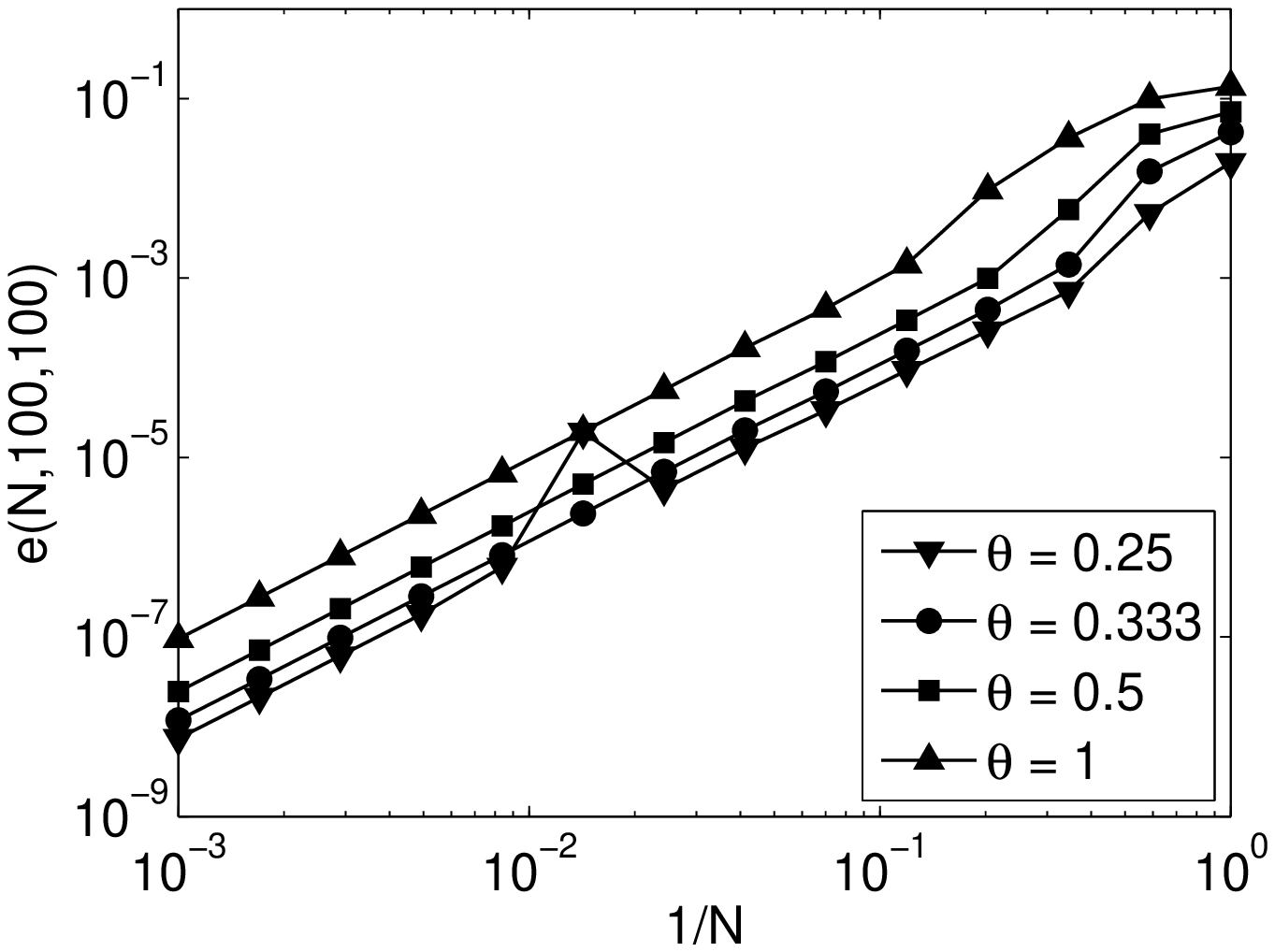}
\includegraphics[scale=0.5]{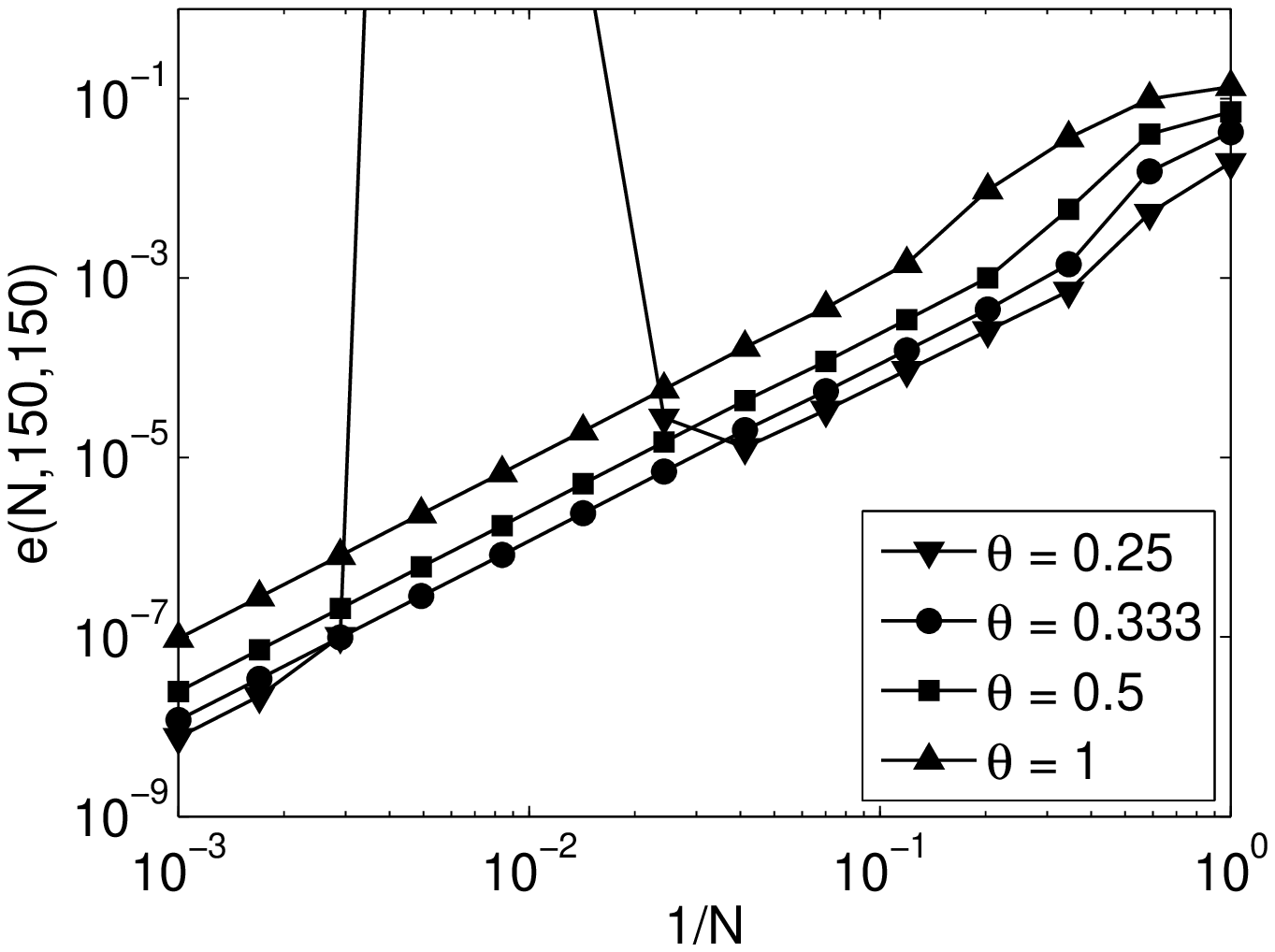}
\includegraphics[scale=0.5]{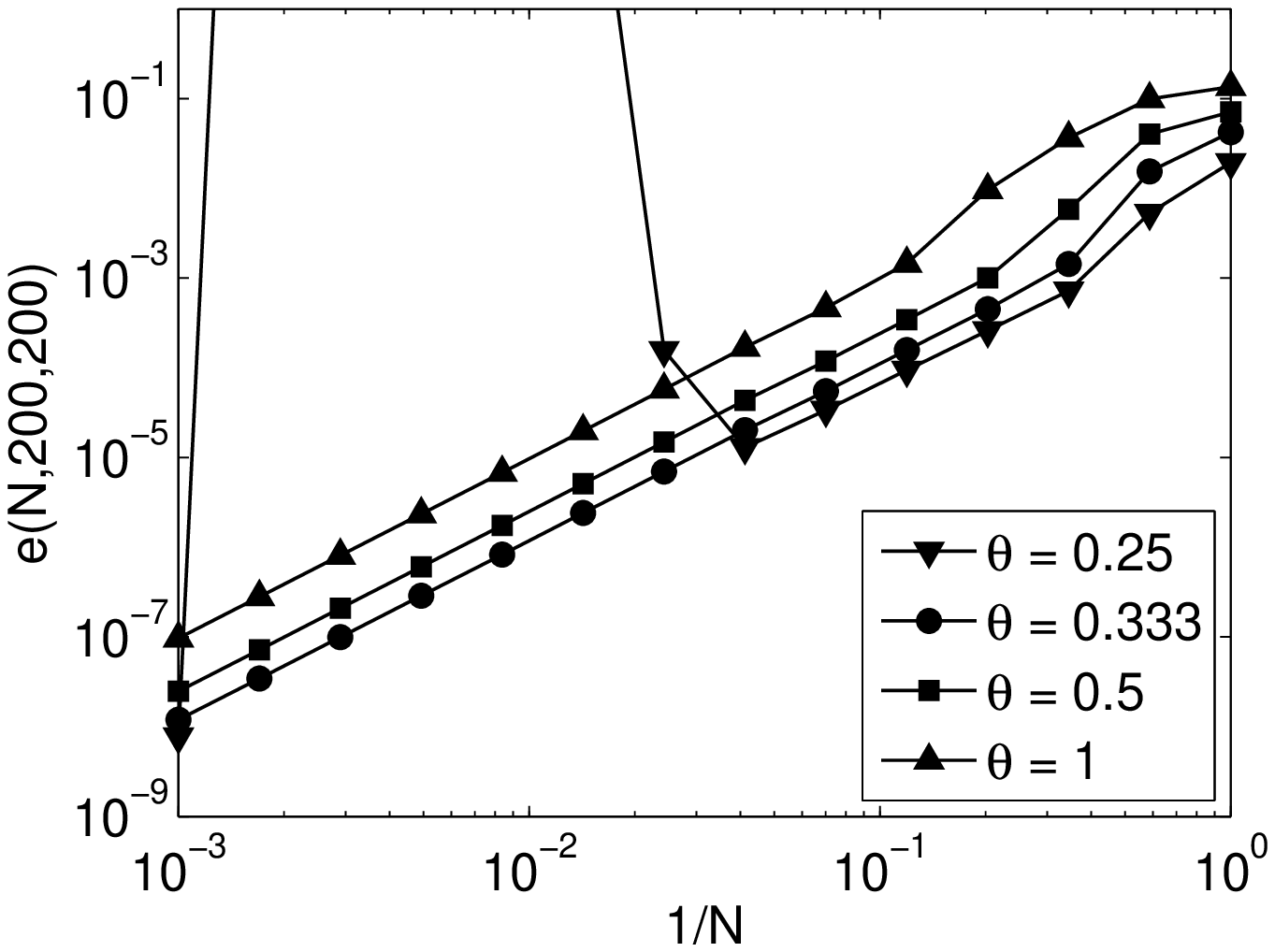}
\caption{Global discretization errors $e(N,m_{1},m_{2})$ versus $1/N$ for $m_{1}=m_{2}=50$ (top left), $m_{1}=m_{2}=100$ (top right), $m_{1}=m_{2}=150$ (bottom left) and $m_{1}=m_{2}=200$ (bottom right) with $\theta = \tfrac{1}{4}, \tfrac{1}{3}, \tfrac{1}{2}, 1$.}
\label{fig:Exp}
\end{center}
\end{figure}

\newpage
\setcounter{equation}{0}
\setcounter{theorem}{0}
\section{Conclusions}\label{concl}
We have proved a useful second-order convergence result for the MCS scheme in the application to two-dimensional time-dependent convection-diffusion equations with mixed derivative term.
Here the obtained bound on the global temporal discretization errors has the important property that it is independent of the (arbitrarily small) spatial mesh width from the semidiscretization.
Based on the convergence analysis in the present paper and the stability results from \cite{IHM11,M14}, we recommend to select, in the application to these equations, the parameter of the MCS scheme such that $\theta \ge \tfrac{1}{3}$.
Numerical experiments further indicate that a smaller parameter value often yields a smaller error constant.
In future research we wish to extend our convergence results, among others, to higher-dimensional problems and to problems with nonsmooth initial functions.

\section*{Acknowledgements}
The second author holds a PhD Fellowship of the Research Foundation--Flanders.

\newpage
\appendix

\setcounter{equation}{0}
\setcounter{theorem}{0}

\section{Proof of Lemma \ref{le:RealFunction}  }
\label{ap:prooflemma}
For the case $\delta = 0$ the result is trivial. Next, consider the case $0 < \delta < 1$.
To find the minimum of $f$ on its domain we determine first its stationary points in $(0,\infty)\times (0,\infty)$.
These are given by
\begin{equation}
\label{eq:afgeleidennul}
\left\{ \begin{array}{l}
f_{x}(x,y) = x \frac{\sqrt{1+y^{2}}}{\sqrt{1+x^{2}}} - \delta = 0, \\\\
f_{y}(x,y) = y \frac{\sqrt{1+x^{2}}}{\sqrt{1+y^{2}}} - \delta = 0.
\end{array} \right.
\end{equation}
From (\ref{eq:afgeleidennul}) it follows that
$$ 
xy = \delta^2,
$$
which yields that $x$ and $y$ are nonzero and
\begin{equation}
\label{eq:solutionsystem}
y = \frac{\delta^{2}}{x}. 
\end{equation}
From (\ref{eq:afgeleidennul}) it also follows that
$$ 
\frac{x}{1+x^{2}} = \frac{y}{1+y^{2}}.
$$
Inserting (\ref{eq:solutionsystem}) into this yields
$$ 
\frac{x}{1+x^{2}} = \frac{x\delta^{2}}{x^{2} + \delta^{4}}, 
$$
which simplifies to
$$ 
\left( 1 - \delta^{2} \right) x^{2} = \delta^{2}(1-\delta^{2}).  
$$
Because
$$
\eta := 
1 -  \delta^{2} >0,
$$
this factor can be divided out.
We thus conclude that
\begin{equation*}
\label{eq:valuex}
x = \delta >0,
\end{equation*} 
and by (\ref{eq:solutionsystem}),
\begin{equation*}
\label{eq:valuex}
y=\delta=x.
\end{equation*} 
Hence, the system (\ref{eq:afgeleidennul}) has precisely one solution, given by
\begin{equation}
\label{eq:minimumpoint}
(x,y) = \left(\delta, \delta  \right).
\end{equation} 
We next prove that $f$ possesses a relative minimum in its stationary point \eqref{eq:minimumpoint} by showing that $f_{xx}>0$, $f_{yy} > 0$ and $f_{xx}f_{yy} - f_{xy}^{2}>0$ in this point. 
For arbitrary $(x,y)$ there holds
\begin{eqnarray*}
f_{xx}(x,y) &=& \frac{\sqrt{1+y^{2}}}{(1+x^{2})^{3/2}} > 0, \\
f_{yy}(x,y) &=& \frac{\sqrt{1+x^{2}}}{(1+y^{2})^{3/2}} > 0, \\
f_{xy}(x,y) &=& \frac{xy}{\sqrt{1+x^{2}}\sqrt{1+y^{2}}}
\end{eqnarray*}
and
$$
(f_{xx}f_{yy} - f_{xy}^{2})(x,y) = \frac{1-x^{2}y^{2}}{(1+x^{2})(1+y^{2})} = \frac{(1-xy)(1+xy)}{(1+x^{2})(1+y^{2})}.
$$
It is therefore sufficient to prove that $1-xy $ is strictly positive in the point (\ref{eq:minimumpoint}) and indeed $1-\delta^{2} = \eta >0$.
Hence, $f$ has a relative minimum in (\ref{eq:minimumpoint}), where it takes the value
\begin{eqnarray*}
\label{eq:minimum}
1 + \delta^{2} -  \delta (\delta + \delta) = \eta > 0.
\end{eqnarray*}
It remains to prove that on the boundary of its domain $f$ is greater than the value $\eta$.
First,
\begin{eqnarray*}
f(x,y) &\geq & \sqrt{1+x^{2}} - \delta(x + y) \\
& \geq & x - \delta(x + y) \\
& = & (1-\delta) x - \delta y.
\end{eqnarray*}
Thus for any given fixed $y \in \R^{+}$ there holds
$$ 
\lim_{x \rightarrow \infty} f(x,y) = \infty. 
$$
Since $f(x,y) = f(y,x)$ for all $(x,y)$ in the domain of $f$, it also holds for any given fixed $x \in \R^{+}$ that
$$ 
\lim_{y \rightarrow \infty} f(x,y) = \infty. 
$$
We finally show that $f$ is always greater than $\eta$ whenever $x=0$ or $y=0$. 
By the same symmetry argument as above, it suffices to consider only $y=0$.
Define 
$$ 
g: \R^{+} \rightarrow \R : x \rightarrow \sqrt{1+x^{2}} - \delta x, 
$$
so that $g(x) = f(x,0)$. 
Then
$$ 
g(0) = 1 > \eta \quad {\rm and} \quad \lim_{x \rightarrow \infty} g(x) \geq \lim_{x \rightarrow \infty} (1-\delta)x = \infty. 
$$
Next,
$$ 
\left\{ \begin{array}{l}
g_{x}(x) = \displaystyle\frac{ x}{\sqrt{1+x^{2}}} - \delta, \\\\
g_{xx}(x) =  \displaystyle\frac{1}{(1+x^{2})^{3/2}} > 0.
\end{array} \right.
$$
Putting $g_{x}(x) = 0$, it readily follows that $g$ has one relative minimum, which is in the point
$$ 
x = \sqrt{\frac{1-\eta}{\eta}},
$$
where it takes the value
$$
g(x) = \sqrt{\eta}.
$$
 Since
$$ 
\min_{x \in \R^{+}} g(x) = \sqrt{\eta} > \eta
$$
the proof is complete for $0<\delta<1$. 
For the case $\delta = 1$ the result of the lemma is easily obtained by a continuity argument.
\begin{flushright}
\mbox{\tiny $\blacksquare$}
\end{flushright}

\vfill \eject

\end{document}